\documentclass[11pt]{amsart}

\newtheorem{theorem}{Theorem}[section]

\newtheorem{cor}[theorem]{Corollary}

\theoremstyle{definition}
\newtheorem{definition}[theorem]{Definition}
\newtheorem{example}[theorem]{Example}

\theoremstyle{remark}
\newtheorem{remark}[theorem]{Remark}

\numberwithin{equation}{section}

\begin{document}

\newcommand{\spacing}[1]{\renewcommand{\baselinestretch}{#1}\large\normalsize}
\spacing{1.14}

\title{Invariant Matsumoto metrics on homogeneous spaces}

\author {H. R. Salimi Moghaddam}

\address{Department of Mathematics, Faculty of  Sciences, University of Isfahan, Isfahan,81746-73441-Iran.} \email{salimi.moghaddam@gmail.com and hr.salimi@sci.ui.ac.ir}



\begin{abstract}
In this paper we consider invariant Matsumoto metrics which are
induced by invariant Riemannian metrics and invariant vector
fields on homogeneous spaces then we give the flag curvature
formula of them. Also we study the special cases of naturally
reductive spaces and bi-invariant metrics. We end the article by
giving some examples of geodesically complete Matsumoto spaces.
\end{abstract}

\maketitle


\section{\textbf{Introduction}}\label{intro}
In the last decade the study of invariant Finsler structures on
Lie groups and homogeneous spaces has been extended. Lie groups
and homogeneous spaces equipped with invariant Finsler metrics are
best spaces for finding spaces with some curvature properties.
Some curvature properties of these manifolds have been studied  in
\cite{DeHo1}, \cite{DeHo2}, \cite{EsSa}, \cite{Sa1}, \cite{Sa2}
and \cite{Sa3}.\\
An important family of Finsler metrics is the family of
$(\alpha,\beta)-$metrics. These metrics are introduced by M.
Matsumoto (see \cite{Ma2}). The interesting and important examples
of $(\alpha,\beta)-$metrics are Randers metric $\alpha+\beta$,
Kropina metric $\frac{\alpha^2}{\beta}$, and Matsumoto metric
$\frac{\alpha^2}{\alpha-\beta}$, where
$\alpha(x,y)=\sqrt{g_{ij}(x)y^iy^j}$ and $\beta(x,y)=b_i(x)y^i$
and $g$ and $b$ are a Riemannian metric and  a 1-form respectively
as follows:
\begin{eqnarray}
  g&=&g_{ij}dx^i\otimes dx^j \\
  b&=&b_idx^i.
\end{eqnarray}\\
In the Matsumoto metric, the 1-form $b=b_idx^i$ was originally to
be induced by the Earth's gravity (see \cite{AnInMa} or
\cite{Ma1}.).\\
In a natural way, the Riemannian metric $g$ induces an inner
product on any cotangent space $T^\ast_xM$ such that
$<dx^i(x),dx^j(x)>=g^{ij}(x)$. The induced inner product on
$T^\ast_xM$ induces a linear isomorphism between $T^\ast_xM$ and
$T_xM$ (for more details see \cite{DeHo2}.). Then the 1-form $b$
corresponds to a vector field $\tilde{X}$ on $M$ such that
\begin{eqnarray}
  g(y,\tilde{X}(x))=\beta(x,y).
\end{eqnarray}

Therefore we can write the Matsumoto metric
$F=\frac{\alpha^2}{\alpha-\beta}$ as follows:
\begin{eqnarray}\label{F}
  F(x,y)=\frac{\alpha(x,y)^2}{\alpha(x,y)-g(\tilde{X}(x),y)}.
\end{eqnarray}
One of the fundamental quantities which associates with a Finsler
space is flag curvature. Flag curvature is computed by the
following formula:
\begin{eqnarray}\label{flag}
  K(P,Y)=\frac{g_Y(R(U,Y)Y,U)}{g_Y(Y,Y).g_Y(U,U)-g_Y^2(Y,U)},
\end{eqnarray}
where $g_Y(U,V)=\frac{1}{2}\frac{\partial^2}{\partial s\partial
t}(F^2(Y+sU+tV))|_{s=t=0}$, $P=span\{U,Y\}$,
$R(U,Y)Y=\nabla_U\nabla_YY-\nabla_Y\nabla_UY-\nabla_{[U,Y]}Y$ and
$\nabla$ is the Chern connection induced by $F$ (see \cite{BaChSh}
and \cite{Sh}.).\\
In this paper we consider invariant Matsumoto metrics which are
induced by invariant Riemannian metrics and invariant vector
fields on homogeneous spaces then we give the flag curvature
formula of them. Also we study the special cases of naturally
reductive spaces and bi-invariant metrics. We end the article by
giving some examples of geodesically complete Matsumoto spaces.

\section{\textbf{Flag curvature of invariant Matsumoto metrics on homogeneous spaces}}

Let $G$ be a compact Lie group, $H$ a closed subgroup, and
$\frak{g}$ and $\frak{h}$ be the Lie algebras of $G$ and $H$
respectively. Suppose that $g_0$ is a bi-invariant Riemannian
metric on $G$, then the tangent space of the homogeneous space
$G/H$ is given by the orthogonal compliment $\frak{m}$ of
$\frak{h}$ in $\frak{g}$ with respect to $g_0$. Each invariant
metric $g$ on $G/H$ is determined by its restriction to
$\frak{m}$. The arising $Ad_H$-invariant inner product from $g$ on
$\frak{m}$ can extend to an $Ad_H$-invariant inner product on
$\frak{g}$ by taking $g_0$ for the components in $\frak{h}$. In
this way the invariant metric $g$ on $G/H$ determines a unique
left invariant metric on $G$ that we also denote by $g$. The
values of $g_0$ and $g$ at the identity are inner products on
$\frak{g}$ and we determine them by $<.,.>_0$ and $<.,.>$
respectively. The inner product $<.,.>$ determines a positive
definite endomorphism $\phi$ of $\frak{g}$ such that $<X,Y>=<\phi
X,Y>_0$ for all $X, Y\in\frak{g}$.\\

\begin{theorem}\label{flagcurvature}
Let $G, H, \frak{g}, \frak{h}, g, g_0$ and $\phi$ be as above.
Assume that $\tilde{X}$ is an invariant vector field on $G/H$
which is parallel with respect to $g$ and
$\sqrt{g(\tilde{X},\tilde{X})}<\frac{1}{2}$ and $X:=\tilde{X}_H$. Suppose that
$F=\frac{\alpha^2}{\alpha-\beta}$ is the Matsumoto metric induced
by $g$ and $\tilde{X}$. Assume that $(P,Y)$ is a flag in
$T_H(G/H)$ such that $\{Y,U\}$ is an orthonormal basis of $P$ with
respect to $<.,.>$. Then the flag curvature of the flag $(P,Y)$ in
$T_H(G/H)$ is given by
\begin{equation}\label{main-flag-cur-formula}
    K(P,Y)=\frac{(1-<Y,X>)^2\{B(1-<Y,X>)(1-2<Y,X>)+3A<U,X>\}}{(1-<Y,X>)(1-2<Y,X>)+2<U,X>^2},
\end{equation}
where
\begin{eqnarray}\label{A}
   A&=&<R(U,Y)Y,X>\nonumber\\
   &=&-\frac{1}{4}(<[\phi U,Y]+[U,\phi Y],[Y,X]>_0+<[U,Y],[\phi Y,X]+[Y,\phi X]>_0)\nonumber\\
    &&-\frac{3}{4}<[Y,U],[Y,X]_\frak{m}>-\frac{1}{2}<[U,\phi X]+[X,\phi U],\phi^{-1}([Y,\phi Y])>_0\\
    &&+\frac{1}{4}<[U,\phi Y]+[Y,\phi U],\phi^{-1}([Y,\phi X]+[X,\phi Y])>_0,\nonumber
\end{eqnarray}
and
\begin{eqnarray}\label{B}
  B&=&<R(U,Y)Y,U>\nonumber\\
  &=&-\frac{1}{2}<[\phi U,Y]+[U,\phi Y],[Y,U]>_0\nonumber \\
  && -\frac{3}{4}<[Y,U],[Y,U]_{\frak{m}}>-<[U,\phi U],\phi^{-1}([Y,\phi Y])>_0 \\
  && +\frac{1}{4}<[U,\phi Y]+[Y,\phi U],\phi^{-1}([Y,\phi U]+[U, \phi Y])>_0.\nonumber
\end{eqnarray}
\end{theorem}

\begin{proof}
From the assumption, $\tilde{X}$ is parallel with respect to $g$ therefore the Chern
connection of $F$ coincides on the Levi-Civita connection of $g$
(see \cite{AnInMa}.). So the Finsler metric $F$ and the Riemannian
metric $g$ have the
same curvature tensor. We show it by $R$.\\
By using the definition of $g_Y(U,V)$ and some computations for
$F$ we have:
\begin{eqnarray}\label{g_Y}
  g_Y(U,V)&=&\frac{1}{(\sqrt{g(Y,Y)}-g(Y,X))^2}\{4g(Y,U)g(Y,V)+2g(Y,Y)g(U,V)\}\nonumber\\
  &&+\frac{1}{(\sqrt{g(Y,Y)}-g(Y,X))^4}\{-4g(Y,Y)g(Y,U)g(Y,V)\\
  &&+g(Y,Y)^{\frac{3}{2}}(g(Y,V)g(U,X)+g(Y,U)g(V,X))\nonumber\\
  &&+g(Y,Y)^2(3g(U,X)g(V,X)-g(U,V))\nonumber\\
  &&+\sqrt{g(Y,Y)}g(Y,X)(7g(Y,U)g(Y,V)+g(Y,Y)g(U,V))\nonumber\\
  &&-4g(Y,Y)g(Y,X)(g(Y,V)g(U,X)+g(Y,U)g(V,X))\nonumber\}
\end{eqnarray}
Now by using the above formula and the fact that $\{Y,U\}$ is an
orthonormal basis for $P$ with respect to $g$, we have
\begin{eqnarray}\label{eq1}
  g_Y(R(U,Y)Y,U)&=&\frac{2<R(U,Y)Y,U>}{(1-<Y,X>)^2}\nonumber\\
  &&+\frac{1}{(1-<Y,X>)^4}\{<R(U,Y)Y,Y><U,X>\nonumber\\
  &&+3<R(U,Y)Y,X><U,X>-<R(U,Y)Y,U>\\
  &&+<Y,X><R(U,Y)Y,U>\nonumber\\
  &&-4<Y,X><R(U,Y)Y,Y><U,X>\}\nonumber
\end{eqnarray}
and
\begin{eqnarray}\label{eq2}
  g_Y(Y,Y).g_Y(U,U)-g^2_Y(U,Y)&=& \frac{2}{(1-<Y,X>)^4}+\frac{2<U,X>^2+<Y,X>-1}{(1-<Y,X>)^6}.
\end{eqnarray}
We can obtain the relations (\ref{A}) and (\ref{B}) by using
P\"uttmann's formula (see \cite{Pu}.).\\
Substituting the relations (\ref{eq1}) and (\ref{eq2}) in the equation
(\ref{flag}) completes the proof.

\end{proof}
\begin{remark}
A homogeneous space $M=G/H$ with a $G-$invariant indefinite
Riemannian metric $g$ is said to be naturally reductive if it
admits an $ad(H)$-invariant decomposition
$\frak{g}=\frak{h}+\frak{m}$ satisfying the condition
\begin{eqnarray}
 B(X,[Z,Y]_{\frak{m}})+B([Z,X]_{\frak{m}},Y)=0 \hspace{1.5cm}\mbox{for} \ \ \ X, Y, Z \in
 \frak{m},
\end{eqnarray}
where $B$ is the bilinear form on $\frak{m}$ induced by $\frak{g}$
and $[,]_{\frak{m}}$ is the projection to $\frak{m}$ with respect
to the decomposition $\frak{g}=\frak{h}+\frak{m}$ (For more
details see \cite{KoNo}.). In this case the relation
(\ref{main-flag-cur-formula}) for the flag curvature reduces to a
simpler equation, because in the case of naturally reductive
homogeneous space we have (see \cite{KoNo}.)

\begin{eqnarray}
    R(U,Y)Y&=&\frac{1}{4}[Y,[U,Y]_{\frak{m}}]_{\frak{m}}+[Y,[U,Y]_{\frak{h}}].
\end{eqnarray}
\end{remark}

Now we consider the case that the invariant Matsumoto metric is
defined by a bi-invariant Riemannian metric on a Lie group.
\begin{theorem}
Let $G$ be a Lie group and $g$ be a bi-invariant Riemannian metric
on $G$. Assume that $\tilde{X}$ is a left invariant vector field
on $G$ which is parallel with respect to $g$ and
$\sqrt{g(\tilde{X},\tilde{X})}<\frac{1}{2}$ and $X:=\tilde{X}_H$. Suppose that
$F=\frac{\alpha^2}{\alpha-\beta}$ is the Matsumoto metric induced
by $g$ and $\tilde{X}$, also let $(P,Y)$ be a flag in $T_eG$ such
that $\{Y,U\}$ be an orthonormal basis of $P$ with respect to
$<.,.>$. Then the flag curvature of the flag $(P,Y)$ in $T_eG$ is
given by
\begin{eqnarray}\label{flag bi-invariant}
   K(P,Y)&=&\frac{-(1-<Y,X>)^2}{4(1-<Y,X>)(1-2<Y,X>)+8<U,X>^2}\nonumber\\
   &&\hspace{-1cm} \{<[[U,Y],Y],U>(1-<Y,X>)(1-2<Y,X>)+3<[[U,Y],Y],X><U,X>\}.
\end{eqnarray}

\end{theorem}

\begin{proof}
$g$ is bi-invariant therefore we have
$R(U,Y)Y=-\frac{1}{4}[[U,Y],Y]$. Now by using theorem
(\ref{main-flag-cur-formula}), the proof is completed.
\end{proof}

\section{\textbf{Some examples of geodesically complete Matsumoto spaces}}

In this section we give some examples of geodesically complete
Matsumoto spaces. We begin with a definition from \cite{BeEhEa}.

\begin{definition}
The Riemannian manifold $(M,g)$ is said to be homogeneous if the
group of isometries of $M$ acts transitively on $M$.
\end{definition}

\begin{theorem}
Suppose that $(M,g)$ is a homogeneous Riemannian manifold. Let $F$
be a Matsumoto metric of Berwald type defined by $g$ and a 1-form
$b$. Then $(M,F)$ is geodesically complete. Moreover if $M$ is
connected then $(M,F)$ is complete.
\end{theorem}

\begin{proof}
The Chern connection of $F$ and the Levi-Civita connection of $g$
coincide and therefore their geodesics coincide too. On the other
hand $(M,g)$ is a homogeneous Riemannian manifold, hence $(M,g)$
is geodesically complete (see \cite{BeEhEa} page 185.). Therefore
$(M,F)$ is geodesically complete. If $M$ is connected then by
using Hofp-Rinow theorem for Finsler manifolds, $(M,F)$ is
complete.
\end{proof}

\begin{cor}
Let $G$ be a Lie group and $g$ be a left invariant Riemannian
metric on $G$. Also suppose that $X$ is a parallel vector field
with respect to the Levi-Civita connection of $g$ such that
$\sqrt{g(X,X)}<\frac{1}{2}$. Then the Matsumoto metric defined by $g$, $X$ and the
relation (\ref{F}) is geodesically complete.
\end{cor}

Now we consider an abelian Lie group equipped with a left invariant
Riemannian metric. We know that this space is flat. In this
case we have the following theorem,

\begin{theorem}
Let $G$ be an abelian Lie group equipped with a left invariant
Riemannian metric $g$ and let $\frak{g}$ be the Lie algebra of
$G$. Suppose that $X\in\frak{g}$ is a left invariant vector field
with $\sqrt{g(X,X)}<\frac{1}{2}$. Then the Matsumoto metric $F$ defined by
the formula (\ref{F}) is a flat geodesically complete locally
Minkowskian metric on $G$.
\end{theorem}

\begin{proof}
Assume that $U,V,W\in\frak{g}$, now by using the Koszul's formula
and the fact that $G$ is abelian we have $\nabla_YX=0$, for any
$Y\in\frak{g}$. Hence $X$ is parallel with respect to $\nabla$ and
$F$ is of Berwald type. Also the curvature tensor $R=0$ of $g$
coincides on the curvature tensor of $F$ and therefore the flag
curvature of $F$ is zero. $F$ is a flat Berwald metric therefore
by proposition 10.5.1 (page 275) of \cite{BaChSh}, $F$ is locally
Minkowskian.
\end{proof}

\begin{example}($E(2)$ group of rigid motions of Euclidean
2-space). We consider the Lie group $E(2)$ as follows:
\begin{eqnarray}
  E(2)=\{\left[%
\begin{array}{ccc}
  \cos\theta & -\sin\theta & a \\
  \sin\theta & \cos\theta & b \\
  0 & 0 & 1 \\
\end{array}%
\right]| a,b,\theta\in\Bbb{R}\}.
\end{eqnarray}

The Lie algebra of $E(2)$ is of the form
\begin{eqnarray}
  \frak{e}(2)= span\{x=\left[%
\begin{array}{ccc}
  0 & 0 & 1 \\
  0 & 0 & 0 \\
  0 & 0 & 0 \\
\end{array}%
\right],y=\left[%
\begin{array}{ccc}
  0 & 0 & 0 \\
  0 & 0 & 1 \\
  0 & 0 & 0 \\
\end{array}%
\right],z=\left[%
\begin{array}{ccc}
  0 & -1 & 0 \\
  1 & 0 & 0 \\
  0 & 0 & 0 \\
\end{array}%
\right]\},
\end{eqnarray}
where
\begin{eqnarray}
  [x,y]=0 \ \ \ , \ \ \ [y,z]=x\ \ \ , \ \ \ [z,x]=y.
\end{eqnarray}
Now let $g$ be the left invariant Riemannian metric induced by the
following inner product,
\begin{eqnarray}\label{inner product}
  <x,x>=<y,y>=<z,z>=\lambda^2 \ \ , \ \ <x,y>=<y,z>=<z,x>=0, \ \ \lambda>0.
\end{eqnarray}

In \cite{Sa3} we showed that the left invariant vector fields
which are parallel with respect to the Levi-Civita connection of
this space are of the form $U=uz$. Also we proved that $R=0$.
Assume that $\sqrt{<U,U>}<\frac{1}{2}$, in other words let
$0<|u|<\frac{1}{2\lambda}$. Hence, the left invariant Matsumoto
metric $F$ defined by $g$ and $U$ with formula (\ref{F}) is of
Berwald type. Also since $F$ is of Berwald type therefore the
curvature tensor of $F$ and $g$ coincide and $F$ is of zero
constant flag curvature. Hence $F$ is locally Minkowskian.
\end{example}
\begin{example} Another example of flat geodesically complete locally
Minkowskian Matsumoto spaces is described as follows.\\

Let $\frak{g}=span\{x,y,z\}$ be a Lie algebra such that
\begin{eqnarray}
  [x,y]=\alpha y+\alpha z \ \ \ , \ \ \ [y,z]=2\alpha x \ \ \ , \ \ \ [z,x]=\alpha
  y+\alpha z \ \ \ , \ \ \ \alpha\in\Bbb{R}.
\end{eqnarray}
Also consider the inner product described by (\ref{inner product})
on $\frak{g}$.\\
Suppose that $G$ is a Lie group with Lie algebra $\frak{g}$, and
$g$ is the left invariant Riemannian metric induced by the above
inner product $<.,.>$ on $G$.\\
A direct computation show that $R=0$, therefore $(G,g)$ is a flat
Riemannian manifold. Also in \cite{Sa3} we proved that vector
fields which are parallel with respect to the Levi-Civita
connection of $(G,g)$ are of the form $U=uy-uz$. Now suppose that
$\sqrt{2}|u|\lambda=\sqrt{<U,U>}<\frac{1}{2}$ or equivalently let
$0<|u|<\frac{1}{2\sqrt{2}\lambda}$. Therefore the invariant
Matsumoto metric $F$ defined by $g$ and $U$ is a flat geodesically
complete locally Minkowskian metric on $G$. Also if we consider
$G$ is connected, $(G,F)$ will be complete.
\end{example}

{\large{\textbf{Acknowledgment.}}} This research was supported by the Center of Excellence for 
Mathematics at the University of Isfahan.

\bibliographystyle{amsplain}

\end{document}